\documentclass[10pt,letterpaper]{amsart}
\usepackage{amsmath,amssymb,amsfonts,amsthm}
\usepackage{graphicx}
\usepackage{booktabs}
\usepackage{array}
\usepackage{bm}
\usepackage{enumitem}
\usepackage{mathrsfs}
\usepackage{cite}
\usepackage{hyperref}
\usepackage{cleveref}
\usepackage{doi}

\numberwithin{equation}{section}

\newtheorem{theorem}{Theorem}[section]
\newtheorem{proposition}[theorem]{Proposition}

\newtheorem{corollary}[theorem]{Corollary}
\theoremstyle{definition}
\newtheorem{definition}[theorem]{Definition}
\newtheorem{remark}[theorem]{Remark}
\newtheorem{example}[theorem]{Example}

\title[Time-Scaled Intertwining Cocycles]
{Time-Scaled Intertwining Cocycles and Identifiability of Multi-Semigroup Mixtures on Hilbert Operator Networks}

\author{Anton Alexa}
\address{Independent Researcher, Chernivtsi, Ukraine}
\email{mail@antonalexa.com}

\begin{document}

\begin{abstract}
We study rigidity phenomena for time-scaled intertwining families of dissipative semigroups $\mathcal S_i(t)=e^{-tA_i}$ and prove that a network of bounded injective operators satisfying $K_{ij}\mathcal S_j(t)=\mathcal S_i(\lambda_{ij}t)K_{ij}$ and $K_{ik}=K_{ij}K_{jk}$ necessarily admits a multiplicative gauge representation $\lambda_{ij}=\tau_i/\tau_j$, if and only if the renormalized generators $\{\tau_iA_i\}$ form a common isospectral class with matching eigenspace dimensions; in particular, eigenspaces are transported isomorphically across sectors. The operators $K_{ij}$ define parallel transport in a flat Hilbert bundle over the index network, with flatness derived from the intertwining constraints rather than assumed. As an application, the mixture observable $M(t)=\sum_i w_i\mathcal B_0K_{0i}\mathcal S_i(t)\psi_i$ reduces under finite spectral support to a structured exponential sum. Under spectral separation, the modal parameters are uniquely identifiable, with sector tags determined intrinsically by the operator spectra; under eigenspace observability, active state components are uniquely recovered. Finite-window exact reconstruction holds from $2L$ samples, and the stability bound $\|\widehat\Theta-\Theta_\ast\|_{\mathcal X}\le C_{\mathrm{stab}}\kappa_{\mathrm{exp}}\varepsilon$ follows with constants explicitly controlled by the spectral geometry and observability of the network.
\end{abstract}
\maketitle

\section{Introduction}
\label{sec:introduction}

Intertwining relations between semigroups are classical in operator
theory~\cite{hille-phillips,pazy1983,engel-nagel}, but the structure of
intertwining families with nontrivial time rescaling has not been fully
characterized. In this work we study rigidity phenomena for such
time-scaled intertwining networks and show that the associated scaling
structure is necessarily of multiplicative gauge type, leading to a
global isospectral compatibility across sectors.

More precisely, we consider families of contraction semigroups
$\mathcal S_i(t)=e^{-tA_i}$ on Hilbert spaces $\mathcal H_i$ together with
bounded injective operators $K_{ij}$ satisfying the cocycle condition
$K_{ik}=K_{ij}K_{jk}$ and the time-scaled intertwining relation
$K_{ij}\mathcal S_j(t)=\mathcal S_i(\lambda_{ij}t)K_{ij}$. We show that these
constraints force the scaling factors to admit a gauge representation
$\lambda_{ij}=\tau_i/\tau_j$, and that, after renormalization, the generators
$\{\tau_iA_i\}$ form a common isospectral class with matching eigenspace
dimensions. In particular, eigenspaces are transported isomorphically
across sectors. This provides a complete intrinsic characterization of
admissible intertwining networks and identifies a rigidity mechanism
linking cocycle structure and spectral data.

Two features distinguish the present setting from existing frameworks.
First, classical intertwining theory treats a single pair of semigroups
with a common time argument,
$K\mathcal S_B(t)=\mathcal S_A(t)K$~\cite{engel-nagel,haase2006}; here
the transfer operators form a cocycle over a full index network, each
relation carries an independent time rescaling, and it is the interplay
of these two constraints that produces the rigidity. Second, unlike
Prony-type reconstruction methods, which operate directly on exponential
sums~\cite{potts-tasche,plonka-tasche2014,plonka-potts-steidl-tasche2019,batenkov-yomdin},
the cocycle structure enforces spectral compatibility before any
reconstruction step, so that structural admissibility and estimation
error are separated at the level of the model itself.

From a structural viewpoint, the operators $K_{ij}$ define parallel
transport maps in a Hilbert bundle over the index network, and the gauge
representation implies that this bundle is flat. Importantly, flatness is
not imposed but follows from the intertwining relations themselves,
revealing an operator-theoretic origin of the underlying geometric
structure.

Background on semigroup theory and the infinitesimal generator calculus
can be found in~\cite{hille-phillips,pazy1983,engel-nagel,arendt2011}.
Intertwining operators in the classical time-preserving setting appear
in connection with similarity theory and the $H^\infty$-functional
calculus~\cite{haase2006}, but the time-scaled network setting studied
here lies outside that framework.

As an application, we consider multi-sector observables obtained by
projecting the transported semigroup evolutions onto a fixed reference
channel. Under finite spectral support, the resulting model reduces to a
structured exponential sum. We show that spectral separation ensures
uniqueness of the modal decomposition together with intrinsic sector
identification, while an additional observability condition yields
recovery of active eigenspace components. Furthermore, we establish
finite-window exact reconstruction and a quantitative stability estimate
in which the reconstruction error is explicitly controlled by the
spectral geometry and the conditioning of the associated exponential
fitting problem.

\section{Time-Scaled Intertwining Networks}
\label{sec:networks}

Classical intertwining in the single-pair setting requires
$K\mathcal{S}_B(t)=\mathcal{S}_A(t)K$ with a common time
argument~\cite{engel-nagel,haase2006}. Two structural extensions are
imposed here: the transfer operators $\{K_{ij}\}$ satisfy a
multiplicative cocycle identity over the full index
network~\eqref{cocycle}, and each intertwining relation carries an
independent positive rescaling $\lambda_{ij}$~\eqref{intertwining}.
These two constraints are not independent: together they force
$\lambda_{ij}=\tau_i/\tau_j$ for a single family of gauge parameters
$\{\tau_i\}$ (Theorem~\ref{thm:gauge-representation}), and cycle
products of $\lambda_{ij}$ are automatically trivial
(Corollary~\ref{cor:cycle-consistency}).

Let $I$ be an index set.
For each $i\in I$, let (i) $\mathcal H_i$ be a separable Hilbert space,
(ii) $A_i$ be a positive self-adjoint operator on $\mathcal H_i$ with
compact resolvent---which forces $\sigma(A_i)$ to be pure point with
finite multiplicities~\cite{kato1980}---and (iii)
$\mathcal S_i(t)=e^{-tA_i}$ be the associated strongly continuous
contraction semigroup~\cite{pazy1983,engel-nagel}. We assume each
sector is nontrivial: $\sigma(A_i)$ contains at least one strictly
positive eigenvalue.

\begin{definition}[Time-Scaled Intertwining Cocycle]
A family of bounded injective operators
\begin{equation}
K_{ij} : \mathcal H_j \to \mathcal H_i
\end{equation}
is called a time-scaled intertwining cocycle
with scaling factors $\lambda_{ij} > 0$
if for all $i,j,k \in I$:

\begin{align}
K_{ik} &= K_{ij} K_{jk}, \label{cocycle}\\
K_{ij} \mathcal S_j(t)
&= \mathcal S_i(\lambda_{ij} t) K_{ij}
\quad \text{for all } t \ge 0. \label{intertwining}
\end{align}
\end{definition}

\begin{remark}
Condition~\eqref{cocycle} is a cocycle-type transitivity condition on
the system of transfer operators.
In the time-preserving case
$\lambda_{ij}=1$ for all $i,j$, condition~\eqref{intertwining}
reduces to ordinary intertwining $K_{ij}\mathcal S_j(t)=\mathcal
S_i(t)K_{ij}$ and imposes no constraint on any scaling factors;
the gauge rigidity of Theorem~\ref{thm:gauge-representation} is
a consequence of the interplay between~\eqref{cocycle}
and~\eqref{intertwining} when the $\lambda_{ij}$ are allowed to
differ from $1$. Injectivity of each $K_{ij}$ is assumed throughout;
it will be seen in Theorem~\ref{cor:global-spectral-compatibility}
that under the cocycle structure, invertibility follows automatically.
\end{remark}

\begin{proposition}[Multiplicativity of Scaling Factors]
\label{prop:multiplicativity}
Assume the family $\{K_{ij}\}$ satisfies \eqref{cocycle}--\eqref{intertwining}.
Then
\begin{equation}
\lambda_{ik} = \lambda_{ij} \lambda_{jk}
\qquad\text{for all }i,j,k\in I.
\end{equation}
\end{proposition}

\begin{proof}
Let $\phi\in\mathcal{H}_k$ be an eigenvector of $A_k$ with eigenvalue $\alpha_k>0$,
and set $v=K_{ik}\phi\neq0$ (injectivity of $K_{ik}$).
Applying \eqref{intertwining} for $K_{ik}$,
\begin{equation}
\mathcal S_i(\lambda_{ik}t)\,v = K_{ik}\mathcal S_k(t)\phi = e^{-\alpha_k t}v,
\end{equation}
so $v$ is an eigenvector of $A_i$ with eigenvalue $\alpha_k/\lambda_{ik}$.
Using $K_{ik}=K_{ij}K_{jk}$ from \eqref{cocycle} and applying \eqref{intertwining}
first to $K_{jk}$ then to $K_{ij}$,
\begin{equation}
\mathcal S_i(\lambda_{ij}\lambda_{jk}t)\,v
= K_{ij}\mathcal S_j(\lambda_{jk}t)K_{jk}\phi
= K_{ij}K_{jk}\mathcal S_k(t)\phi
= e^{-\alpha_k t}v.
\end{equation}
Hence $v$ is also an eigenvector of $A_i$ with eigenvalue $\alpha_k/(\lambda_{ij}\lambda_{jk})$.
Since $A_i$ is self-adjoint and $\alpha_k>0$, the eigenvalue is unique, so
$\lambda_{ik}=\lambda_{ij}\lambda_{jk}$.
\end{proof}

\begin{theorem}[Gauge representation of scaling factors]
\label{thm:gauge-representation}
Assume the cocycle relations \eqref{cocycle}--\eqref{intertwining} hold.
Then there exists a family $\{\tau_i\}_{i\in I}$ with $\tau_i>0$ such that
\begin{equation}
\lambda_{ij}=\frac{\tau_i}{\tau_j}
\qquad\text{for all }i,j.
\end{equation}
Equivalently, the scaling cocycle is a multiplicative coboundary.
\end{theorem}

\begin{proof}
Fix $i_0\in I$ and define $\tau_i=\lambda_{ii_0}>0$.
By Proposition~\ref{prop:multiplicativity},
\begin{equation}
\lambda_{ii_0}=\lambda_{ij}\lambda_{ji_0}
\end{equation}
for all $i,j\in I$, hence
\begin{equation}
\lambda_{ij}=\frac{\lambda_{ii_0}}{\lambda_{ji_0}}
=\frac{\tau_i}{\tau_j}.
\end{equation}
Therefore $\lambda$ is a multiplicative coboundary.
\end{proof}

\begin{corollary}[Cycle consistency]
\label{cor:cycle-consistency}
For every cycle $i_0\to i_1\to\cdots\to i_m=i_0$,
\begin{equation}
\prod_{r=0}^{m-1}\lambda_{i_{r+1}i_r}=1.
\end{equation}
\end{corollary}

\begin{proof}
Iterating Proposition~\ref{prop:multiplicativity} along the cycle gives
\begin{equation}
\prod_{r=0}^{m-1}\lambda_{i_{r+1}i_r}=\lambda_{i_m i_0}=\lambda_{i_0i_0}.
\end{equation}
From \eqref{cocycle} with $j=k=i$ we obtain $K_{ii}=K_{ii}^2$. Since $K_{ii}$
is injective, $K_{ii}=I_{\mathcal H_i}$. Then \eqref{intertwining} with
$j=i$ yields
\begin{equation}
\mathcal S_i(t)=\mathcal S_i(\lambda_{ii}t)\qquad (t\ge0).
\end{equation}
Applying this identity to an eigenvector of $A_i$ with positive eigenvalue
gives $\lambda_{ii}=1$. Hence the cycle product equals $1$.
\end{proof}

The preceding results are purely at the semigroup level. We now pass
to the infinitesimal generators and derive the spectral consequences
of the intertwining structure.

\section{Generator Relations and Spectral Rigidity}
\label{sec:generator}

Differentiating the semigroup intertwining
relation~\eqref{intertwining} with respect to time yields an algebraic
identity between the generators $A_i$ and $A_j$, a standard passage
from semigroup to generator level~\cite{pazy1983,engel-nagel}.
Multiplicativity of the scaling factors was established at the semigroup
level in Proposition~\ref{prop:multiplicativity}; here we record an
independent generator-level proof and then derive spectral consequences.
The generator identity forces spectral inclusion
$\sigma(A_j)\subseteq\lambda_{ij}\sigma(A_i)$; compare the
intertwining-based spectral calculus in~\cite{haase2006}. In the
all-pairs injective setting, each $K_{ij}$ is automatically invertible
and the inclusion sharpens to equality, placing all rescaled generators
in a common isospectral class
(Theorem~\ref{cor:global-spectral-compatibility}).

\begin{theorem}[Generator Intertwining]
\label{thm:generator}
Let $K_{ij}$ satisfy \eqref{intertwining}.
Then on $\mathrm{Dom}(A_j)$,
\begin{equation}
K_{ij} A_j = \lambda_{ij} A_i K_{ij}.
\end{equation}
\end{theorem}

\begin{proof}
Fix $u\in\mathrm{Dom}(A_j)$. Since $\mathcal S_j$ has generator $-A_j$,
\begin{equation}
\lim_{t\downarrow0}\frac{\mathcal S_j(t)u-u}{t}=-A_j u
\quad\text{in }\mathcal H_j.
\end{equation}
Apply the bounded operator $K_{ij}$:
\begin{equation}
\lim_{t\downarrow0}
K_{ij}\frac{\mathcal S_j(t)u-u}{t}
=-K_{ij}A_j u.
\end{equation}
Using \eqref{intertwining},
\begin{equation}
K_{ij}\frac{\mathcal S_j(t)u-u}{t}
=
\frac{\mathcal S_i(\lambda_{ij}t)K_{ij}u-K_{ij}u}{t}
=
\lambda_{ij}
\frac{\mathcal S_i(\lambda_{ij}t)K_{ij}u-K_{ij}u}{\lambda_{ij}t}.
\end{equation}
Therefore the limit
\begin{equation}
\lim_{s\downarrow0}\frac{\mathcal S_i(s)K_{ij}u-K_{ij}u}{s}
\end{equation}
exists, so $K_{ij}u\in\mathrm{Dom}(A_i)$ and
\begin{equation}
A_iK_{ij}u=\frac{1}{\lambda_{ij}}K_{ij}A_j u.
\end{equation}
Multiplying by $\lambda_{ij}$ yields the claimed identity.
\end{proof}

\begin{remark}[Generator-level proof of multiplicativity]
\label{rem:multiplicativity-gen}
Proposition~\ref{prop:multiplicativity} also follows directly from
Theorem~\ref{thm:generator}: with $v=K_{ik}\phi=K_{ij}K_{jk}\phi$
for an eigenvector $\phi$ of $A_k$ with eigenvalue $\alpha_k>0$,
composing the generator identities $K_{jk}A_k=\lambda_{jk}A_jK_{jk}$
and $K_{ij}A_j=\lambda_{ij}A_iK_{ij}$ gives
$K_{ij}K_{jk}A_k=\lambda_{ij}\lambda_{jk}A_iK_{ij}K_{jk}$.
Since $K_{ik}A_k=\lambda_{ik}A_iK_{ik}$ and $A_iv\neq0$,
one reads off $\lambda_{ik}=\lambda_{ij}\lambda_{jk}$.
\end{remark}

\begin{theorem}[Spectral Rigidity]
\label{thm:spectral_rigidity}
Assume \eqref{intertwining} holds for $K_{ij}$.
Then
\begin{equation}
\sigma(A_j) \subseteq \lambda_{ij} \sigma(A_i).
\end{equation}
If $K_{ij}$ is boundedly invertible, equality holds.
Moreover, eigenspace multiplicities satisfy
\begin{equation}
\dim\ker(A_j-\alpha I)\le
\dim\ker\!\Big(A_i-\frac{\alpha}{\lambda_{ij}}I\Big).
\end{equation}
\end{theorem}

\begin{proof}
Because $A_j$ is self-adjoint with compact resolvent, its spectrum is pure
point with finite multiplicities~\cite{kato1980}. Let $\alpha\in\sigma(A_j)$ and choose
$0\neq\phi\in\ker(A_j-\alpha I)$.
Injectivity of $K_{ij}$ gives $K_{ij}\phi\neq0$. By
Theorem~\ref{thm:generator},
\begin{equation}
A_i(K_{ij}\phi)=\frac{\alpha}{\lambda_{ij}}K_{ij}\phi,
\end{equation}
so $\alpha/\lambda_{ij}\in\sigma_p(A_i)\subset\sigma(A_i)$.
Hence $\sigma(A_j)\subseteq \lambda_{ij}\sigma(A_i)$.

For multiplicities, the restriction
\begin{equation}
K_{ij}:\ker(A_j-\alpha I)\to
\ker\!\Big(A_i-\frac{\alpha}{\lambda_{ij}}I\Big)
\end{equation}
is injective, therefore
\begin{equation}
\dim\ker(A_j-\alpha I)\le
\dim\ker\!\Big(A_i-\frac{\alpha}{\lambda_{ij}}I\Big).
\end{equation}

If $K_{ij}$ is boundedly invertible, apply the same argument to
$K_{ij}^{-1}$, which satisfies
\begin{equation}
K_{ij}^{-1}\mathcal S_i(t)=\mathcal S_j\!\left(\frac{t}{\lambda_{ij}}\right)K_{ij}^{-1},
\end{equation}
and obtain the reverse inclusion
$\sigma(A_i)\subseteq\lambda_{ij}^{-1}\sigma(A_j)$, i.e. equality.
\end{proof}

\begin{theorem}[Spectral Rigidity of Intertwining Networks]
\label{cor:global-spectral-compatibility}
Assume the cocycle and intertwining relations
\eqref{cocycle}--\eqref{intertwining}. Then each $K_{ij}$ is
boundedly invertible, and there exist $\tau_i>0$ such that for all $i,j$,
\begin{equation}
\sigma(\tau_i A_i)=\sigma(\tau_j A_j).
\end{equation}
All generators belong to a common scaled isospectral class.
\end{theorem}

\begin{proof}
From \eqref{cocycle} with $j=k=i$ we get $K_{ii}=K_{ii}^2$. Since $K_{ii}$ is
injective, $K_{ii}=I_{\mathcal H_i}$. Hence
\begin{equation}
K_{ij}K_{ji}=K_{ii}=I_{\mathcal H_i},
\qquad
K_{ji}K_{ij}=K_{jj}=I_{\mathcal H_j},
\end{equation}
so each $K_{ij}$ is boundedly invertible with inverse $K_{ji}$.
By Theorem~\ref{thm:spectral_rigidity},
\begin{equation}
\sigma(A_j)=\lambda_{ij}\sigma(A_i).
\end{equation}
By Theorem~\ref{thm:gauge-representation}, $\lambda_{ij}=\tau_i/\tau_j$.
Therefore
\begin{equation}
\sigma(A_j)=\frac{\tau_i}{\tau_j}\sigma(A_i),
\end{equation}
and multiplying by $\tau_j$ yields
\begin{equation}
\sigma(\tau_jA_j)=\sigma(\tau_iA_i).
\end{equation}
Since $i,j\in I$ are arbitrary, the claim follows.
\end{proof}

\begin{proposition}[Transport of Eigen-Subspaces]
\label{prop:eigenspace_transport}
Let $\phi \in \mathcal H_j$ satisfy
\begin{equation}
A_j \phi = \alpha \phi.
\end{equation}
Then
\begin{equation}
A_i (K_{ij} \phi)
=
\frac{\alpha}{\lambda_{ij}}\, K_{ij} \phi.
\end{equation}
In particular, if $K_{ij}\phi \neq 0$, the operator $K_{ij}$
maps eigen-subspaces of $A_j$ into eigen-subspaces of $A_i$
with reciprocal scaling of eigenvalues.
\end{proposition}

\begin{proof}
Using Theorem~\ref{thm:generator}, we have
\begin{equation}
K_{ij} A_j = \lambda_{ij} A_i K_{ij}.
\end{equation}
Applying both sides to $\phi$ gives
\begin{equation}
K_{ij} (\alpha \phi)
=
\lambda_{ij} A_i (K_{ij} \phi),
\end{equation}
hence
\begin{equation}
A_i (K_{ij} \phi)
=
\frac{\alpha}{\lambda_{ij}} K_{ij} \phi.
\end{equation}
\end{proof}

\begin{remark}
Proposition~\ref{prop:eigenspace_transport} says more than spectral
inclusion: each $K_{ij}$ acts as an injective (and, by
Theorem~\ref{cor:global-spectral-compatibility}, invertible)
intertwiner between the eigenspaces of $A_j$ and the corresponding
scaled eigenspaces of $A_i$. This transport of eigenspaces is the
structural fact that will underpin the identifiability analysis in
Section~\ref{sec:identifiability}: distinct sectors contribute spectrally
separated components to the mixture observable precisely because their
eigenspaces are linked by the cocycle geometry, not by coincidence.
\end{remark}

The following converse shows that spectral coincidence is not only
necessary but also sufficient for the existence of an admissible cocycle
network, completing the characterization.

\begin{theorem}[Existence Characterization]
\label{thm:existence-characterization}
Let $\{A_i\}_{i\in I}$ be positive self-adjoint operators with compact
resolvent on separable Hilbert spaces $\{\mathcal H_i\}_{i\in I}$.
The following are equivalent.

\smallskip
\noindent\emph{(i)} There exist positive scalars $\{\tau_i\}_{i\in I}$ and a family of
bounded invertible operators $\{K_{ij}\}_{i,j\in I}$ satisfying the
cocycle~\eqref{cocycle} and intertwining~\eqref{intertwining} relations.

\smallskip
\noindent\emph{(ii)} There exist positive scalars $\{\tau_i\}_{i\in I}$ such that
$\sigma(\tau_i A_i)=\sigma(\tau_j A_j)$ for all $i,j\in I$, and for
every $\alpha\in\sigma(\tau_i A_i)$,
\begin{equation}
  \dim\ker\!\bigl(A_i-\tau_i^{-1}\alpha\,I\bigr)
  =\dim\ker\!\bigl(A_j-\tau_j^{-1}\alpha\,I\bigr)
  \qquad\text{for all }i,j\in I.
\end{equation}

\smallskip
When all eigenvalues of $A_i$ are simple, condition~\emph{(ii)} reduces to the
single spectral equality $\sigma(\tau_i A_i)=\sigma(\tau_j A_j)$.
\end{theorem}

\begin{proof}
\emph{(i)$\Rightarrow$(ii).}
Spectral equality follows from
Theorem~\ref{cor:global-spectral-compatibility}.
Since $K_{ij}$ is invertible and maps $\ker(A_j-\tau_j^{-1}\alpha I)$
into $\ker(A_i-\tau_i^{-1}\alpha I)$ by
Proposition~\ref{prop:eigenspace_transport}, while $K_{ji}=K_{ij}^{-1}$
maps in the reverse direction, the two eigenspaces are isomorphic and
therefore have equal dimension.

\emph{(ii)$\Rightarrow$(i).}
Fix $\{\tau_i\}$ as in~(ii).
For each $\alpha\in\sigma(\tau_i A_i)$ let
$E_{i,\alpha}=\ker(A_i-\tau_i^{-1}\alpha I)$.
Fix a reference index $0\in I$ and for every $i\in I$ and $\alpha$
choose a unitary isomorphism
$V_{i,\alpha}\colon E_{0,\alpha}\to E_{i,\alpha}$ (possible since
$\dim E_{i,\alpha}=\dim E_{0,\alpha}$ by hypothesis).
Define
\begin{equation}
  K_{ij}:=\bigoplus_{\alpha\in\sigma(\tau_i A_i)}
           V_{i,\alpha}\circ V_{j,\alpha}^{-1}
  \colon\mathcal H_j\to\mathcal H_i.
\end{equation}
Since $\{E_{j,\alpha}\}_\alpha$ is an orthogonal decomposition of
$\mathcal H_j$, the operator $K_{ij}$ is bounded and unitary, hence
invertible with $K_{ij}^{-1}=K_{ji}$.

\textit{Cocycle.}
For $\phi\in E_{k,\alpha}$:
$K_{ij}K_{jk}\phi
 =V_{i,\alpha}V_{j,\alpha}^{-1}V_{j,\alpha}V_{k,\alpha}^{-1}\phi
 =V_{i,\alpha}V_{k,\alpha}^{-1}\phi
 =K_{ik}\phi$.

\textit{Intertwining.}
For $\phi\in E_{j,\alpha}$, so $A_j\phi=\tau_j^{-1}\alpha\phi$,
\begin{equation}
  K_{ij}\mathcal S_j(t)\phi
  =e^{-t\tau_j^{-1}\alpha}K_{ij}\phi
  =e^{-(\tau_i/\tau_j)t\cdot\tau_i^{-1}\alpha}K_{ij}\phi
  =\mathcal S_i\!\left(\tfrac{\tau_i}{\tau_j}t\right)K_{ij}\phi.
\end{equation}
Since eigenvectors of $A_j$ span $\mathcal H_j$ (compact resolvent),
the identity extends to all of $\mathcal H_j$ by density.
\end{proof}

With the spectral algebra of the cocycle established, we turn to the
observable built from the mixture of semigroup evolutions.

\section{Multi-Semigroup Mixture Observables}
\label{sec:mixtures}

The cocycle operators $\{K_{0i}\}$ map each sector's evolution into a
common output space, enabling a single observable to aggregate
contributions from all sectors. Proposition~\ref{prop:mixture-expansion}
gives the modal expansion of this mixture as an infinite exponential
series; under finite spectral support it collapses to a finite
exponential sum, the form taken up in Section~\ref{sec:identifiability}.

Fix a reference sector, indexed by $0$.

Let $\mathcal B_0 : \mathcal H_0 \to \mathcal Y$
be a bounded observation operator. For each $i=1,\dots,N$, let
$\psi_i\in\mathcal H_i$ and $w_i>0$, and define
\begin{equation}
M(t)
=
\sum_{i=1}^N
w_i\,
\mathcal B_0
\left(
K_{0i}\mathcal S_i(t)\psi_i
\right),
\qquad t\ge 0.
\label{mixture}
\end{equation}

Let $\{(\alpha_{i,n},\phi_{i,n})\}_{n\ge1}$ be an orthonormal eigenbasis of
$A_i$, so $A_i\phi_{i,n}=\alpha_{i,n}\phi_{i,n}$, $\alpha_{i,n}>0$ and
$\alpha_{i,n}\to\infty$. Write
\begin{equation}
\psi_i=\sum_{n\ge1}\xi_{i,n}\phi_{i,n},
\qquad
\xi_{i,n}:=\langle\psi_i,\phi_{i,n}\rangle_{\mathcal H_i}.
\end{equation}
Define
\begin{equation}
\mu_{i,n}:=\alpha_{i,n},
\qquad
b_{i,n}:=\mathcal B_0K_{0i}\phi_{i,n}\in\mathcal Y.
\end{equation}

\begin{proposition}[Modal expansion of the mixture observable]
\label{prop:mixture-expansion}
For every $t\ge0$,
\begin{equation}
M(t)=\sum_{i=1}^N\sum_{n\ge1}
 w_i\,\xi_{i,n}\,e^{-\mu_{i,n}t}\,b_{i,n}
\quad\text{in }\mathcal Y.
\label{eq:mixture-modal-expansion}
\end{equation}
For every $\tau>0$, the series in \eqref{eq:mixture-modal-expansion}
converges uniformly in $t\in[\tau,\infty)$ in $\mathcal Y$.
\end{proposition}

\begin{proof}
For each fixed $i$, spectral calculus for self-adjoint operators with
compact resolvent gives~\cite{pazy1983,engel-nagel}
\begin{equation}
\mathcal S_i(t)\psi_i
=
\sum_{n\ge1}e^{-\alpha_{i,n}t}\xi_{i,n}\phi_{i,n}
\quad\text{in }\mathcal H_i,\ t\ge0.
\end{equation}
Applying the bounded map $\mathcal B_0K_{0i}$ yields
\begin{equation}
\mathcal B_0K_{0i}\mathcal S_i(t)\psi_i
=
\sum_{n\ge1}e^{-\alpha_{i,n}t}\xi_{i,n}\,\mathcal B_0K_{0i}\phi_{i,n}
=
\sum_{n\ge1}\xi_{i,n}e^{-\alpha_{i,n}t}b_{i,n}
\end{equation}
in $\mathcal Y$. Summing over $i=1,\dots,N$ gives
\eqref{eq:mixture-modal-expansion}; rewriting the exponents as
$e^{-\mu_{i,n}t}$ uses $\mu_{i,n}=\alpha_{i,n}$.

For uniform convergence on $[\tau,\infty)$, let
\begin{equation}
R_{i,N}(t):=\sum_{n>N}\xi_{i,n}e^{-\alpha_{i,n}t}b_{i,n}.
\end{equation}
Then for all $t\ge\tau$,
\begin{equation}
\|R_{i,N}(t)\|_{\mathcal Y}
\le
\|\mathcal B_0K_{0i}\|
\left\|\sum_{n>N}\xi_{i,n}e^{-\alpha_{i,n}t}\phi_{i,n}\right\|_{\mathcal H_i}
\le
\|\mathcal B_0K_{0i}\|
\left(\sum_{n>N}|\xi_{i,n}|^2e^{-2\alpha_{i,n}\tau}\right)^{1/2}.
\end{equation}
The right-hand side tends to $0$ as $N\to\infty$, so each sector series is
uniformly convergent on $[\tau,\infty)$. Since the number of sectors is finite,
the full double series is uniformly convergent on $[\tau,\infty)$.
\end{proof}

\begin{corollary}[Finite exponential model]
\label{cor:finite-mixture-model}
Assume $\mathcal Y=\mathbb C$ and each $\psi_i$ has finite spectral support.
Then \eqref{mixture} reduces to
\begin{equation}
M(t)=\sum_{\ell=1}^L a_\ell e^{-\mu_\ell t},
\end{equation}
for finitely many pairs $(\mu_\ell,a_\ell)$; after merging equal rates,
the exponents $\mu_\ell$ are pairwise distinct.
\end{corollary}

\begin{proof}
Under finite spectral support, each inner series in
\eqref{eq:mixture-modal-expansion} is finite, hence $M$ is a finite linear
combination of exponentials. Grouping equal exponents yields the claimed form
with distinct rates.
\end{proof}

\begin{remark}
Finite spectral support --- each $\psi_i$ being a finite linear combination
of eigenvectors --- is a modeling assumption standard in Prony-type
analysis~\cite{potts-tasche,plonka-tasche2014}. A generic element of
$\mathcal H_i$ carries infinite spectral support, and the infinite series
in Proposition~\ref{prop:mixture-expansion} does not reduce to a finite
exponential sum. Extending the identification and stability theory to the
infinite-dimensional case requires regularization (e.g., truncation with
an a priori mode-count bound) and is left for future work.
\end{remark}

The finite exponential form of Corollary~\ref{cor:finite-mixture-model}
is precisely the structure accessible to Prony-type reconstruction; we
now ask when and how the parameters $\{(\mu_\ell,a_\ell)\}$ can be
recovered uniquely from $M$.

\section{Identifiability of Mixtures}
\label{sec:identifiability}

Identifiability of $M$ has two layers. The first --- uniqueness of the
modal pairs $\{(\mu_\ell,a_\ell)\}$ --- is classical for exponential
sums with distinct rates. The second is specific to the multi-semigroup
structure: spectral separation forces each recovered rate $\mu_\ell$
into exactly one sector, yielding a unique sector assignment $i(\ell)$.
Theorem~\ref{thm:identifiability} establishes both layers and, under
eigenspace observability, recovers the active eigenspace components.

\begin{definition}[Spectral Separation Condition]
The family $\{A_i\}$ satisfies spectral separation if
\begin{equation}
\sigma(A_i)
\cap
\sigma(A_j)
=
\varnothing
\quad \text{for all } i \neq j.
\end{equation}
\end{definition}

\begin{definition}[Sector observability on eigenspaces]
\label{def:sector-observability}
For each sector $i$, let $E_{i,\alpha}=\ker(A_i-\alpha I)$.
We say that $\mathcal B_0K_{0i}$ is \emph{observable on eigenspaces} if
\begin{equation}
\mathcal B_0K_{0i}\big|_{E_{i,\alpha}} \text{ is injective for every }\alpha\in\sigma_p(A_i).
\end{equation}
\end{definition}

\begin{remark}
When $\mathcal Y=\mathbb C$, a linear map $E_{i,\alpha}\to\mathbb C$ cannot be
injective unless $\dim E_{i,\alpha}\le 1$. Definition~\ref{def:sector-observability}
is therefore non-vacuous in the scalar-output case only when every active
eigenvalue is simple ($\dim E_{i,\alpha}=1$); the condition then reduces to
$\mathcal B_0K_{0i}\phi_{i,\alpha}\neq0$.
For eigenvalues of higher multiplicity, one must either take $\mathcal Y$
of dimension at least $\max_\alpha\dim E_{i,\alpha}$, or restrict the
reconstruction target to the one-dimensional projection
$\langle P_{i,\alpha}\psi_i,\phi_{i,\alpha}\rangle\phi_{i,\alpha}$ along a
chosen basis vector $\phi_{i,\alpha}$.
Throughout Theorems~\ref{thm:identifiability} and~\ref{thm:stability},
the scalar case $\mathcal Y=\mathbb C$ is applied under the standing
assumption that all active eigenvalues are simple.
\end{remark}

\begin{theorem}[Identifiability and sector tagging]
\label{thm:identifiability}
Assume $\mathcal Y=\mathbb C$ and that the observable is a finite exponential sum
\begin{equation}
M(t)=\sum_{\ell=1}^L a_\ell e^{-\mu_\ell t},
\qquad \mu_\ell\neq \mu_{\ell'} \ (\ell\neq \ell').
\end{equation}
Assume spectral separation.
Then the modal tuples
\begin{equation}
\{(\mu_\ell,a_\ell)\}_{\ell=1}^L
\end{equation}
are uniquely determined by $M$ (up to permutation of the modal pairs).
Moreover, for each $\ell$ there exists a unique sector index $i(\ell)$ and
unique eigenvalue $\alpha_\ell\in\sigma_p(A_{i(\ell)})$ such that
\begin{equation}
\mu_\ell=\alpha_\ell.
\end{equation}
Hence the tagged family
\begin{equation}
\{(i(\ell),\alpha_\ell,a_\ell)\}_{\ell=1}^L
\end{equation}
is uniquely determined by $M$.

If, in addition, Definition~\ref{def:sector-observability} holds and
$\{w_i\}_{i=1}^N$ is known, then each active eigenspace component
$P_{i,\alpha}\psi_i$ is uniquely recovered from the coefficient attached to
rate $\alpha$.
\end{theorem}

\begin{proof}
Uniqueness of the finite exponential representation follows from linear
independence of exponentials with distinct rates. For completeness, suppose
\begin{equation}
\sum_{\ell=1}^L c_\ell e^{-\mu_\ell t}=0
\quad\text{for all }t\ge0,
\qquad
\mu_1<\cdots<\mu_L.
\end{equation}
Multiplying by $e^{\mu_1 t}$ and sending $t\to\infty$ yields $c_1=0$.
Repeating inductively gives $c_\ell=0$ for all $\ell$, proving linear
independence and hence uniqueness of $\{(\mu_\ell,a_\ell)\}$ up to permutation.

By spectral separation, the sets $\sigma(A_i)$ are pairwise
disjoint. Therefore each recovered rate $\mu_\ell$ belongs to exactly one such
set, which determines a unique sector $i(\ell)$. Inside this sector,
$\alpha_\ell:=\mu_\ell$ is uniquely determined and belongs
to $\sigma_p(A_{i(\ell)})$. This proves uniqueness of the tagged triples.

For the last claim, fix an active pair $(i,\alpha)$ and denote by
$P_{i,\alpha}$ the orthogonal projection onto
$E_{i,\alpha}=\ker(A_i-\alpha I)$. The coefficient attached to rate
$\alpha$ equals
\begin{equation}
a_{i,\alpha}
=
w_i\,\mathcal B_0K_{0i}\!\left(P_{i,\alpha}\psi_i\right).
\end{equation}
If $w_i$ is known and $\mathcal B_0K_{0i}\!\restriction_{E_{i,\alpha}}$ is
injective, then $P_{i,\alpha}\psi_i\in E_{i,\alpha}$ is uniquely determined by
$a_{i,\alpha}$.
\end{proof}

Identifiability guarantees uniqueness of the parameters; the next
question is how to reconstruct them from finitely many samples of $M$.

\section{Finite-Window Reconstruction}
\label{sec:finite-window}

Given $2L$ samples of $M$, the modal data $\{(\mu_\ell,a_\ell)\}$ can
be recovered exactly via a Hankel-matrix method~\cite{potts-tasche,plonka-tasche2014,plonka-potts-steidl-tasche2019}.
The key algebraic fact is that the moment Hankel matrix admits the
factorization $H=VDV^\top$ with $V$ Vandermonde; invertibility follows
from distinctness of rates and non-vanishing of amplitudes.

Let $0<T<\infty$ and assume $M$ is sampled on $[0,T]$.

\begin{theorem}[Finite-window reconstruction]
\label{thm:finite-window}
Assume
\begin{equation}
M(t)=\sum_{\ell=1}^L a_\ell e^{-\mu_\ell t},
\qquad
\mu_\ell>0,\ \mu_\ell\neq\mu_{\ell'}\ (\ell\neq\ell'),\ a_\ell\neq0.
\end{equation}
Then the following hold.
\begin{align*}
\text{(i)}\quad &
\{(\mu_\ell,a_\ell)\}_{\ell=1}^L
\text{ is uniquely determined by }
\{M^{(n)}(0)\}_{n=0}^{2L-1},\\
\text{(ii)}\quad &
\text{for any }h\in(0,T/(2L-1)],
\text{ the same modal data is uniquely determined by }\\
&\{M(0),M(h),\dots,M((2L-1)h)\}.
\end{align*}
\end{theorem}

\begin{proof}
Set
\begin{equation}
m_n:=(-1)^nM^{(n)}(0)=\sum_{\ell=1}^L a_\ell \mu_\ell^n,
\qquad n\ge0.
\end{equation}
Let
\begin{equation}
p(z):=\prod_{\ell=1}^L(z-\mu_\ell)
=z^L+c_{L-1}z^{L-1}+\cdots+c_0.
\end{equation}
Then for every $n\ge0$,
\begin{equation}
\sum_{k=0}^L c_k m_{n+k}=0,
\qquad c_L:=1.
\end{equation}
Hence the coefficient vector $(c_0,\dots,c_{L-1})$ solves the linear system
\begin{equation}
\sum_{k=0}^{L-1} c_k m_{n+k}=-m_{n+L},
\qquad n=0,\dots,L-1.
\end{equation}
Its matrix is the Hankel matrix
\begin{equation}
H=(m_{r+s})_{r,s=0}^{L-1}.
\end{equation}
Write $V=(\mu_\ell^{\,r})_{r=0,\dots,L-1;\,\ell=1,\dots,L}$ and
$D=\mathrm{diag}(a_1,\dots,a_L)$. Then
\begin{equation}
H=VDV^\top.
\end{equation}
Since the $\mu_\ell$ are pairwise distinct, $V$ is invertible; since $a_\ell\neq0$,
$D$ is invertible. Therefore $H$ is invertible, so $(c_0,\dots,c_{L-1})$ is
uniquely determined by $m_0,\dots,m_{2L-1}$.
Consequently $p$ is uniquely determined, and its roots are exactly the rates
$\{\mu_\ell\}_{\ell=1}^L$.

With $\mu_\ell$ known, the amplitudes are the unique solution of the
Vandermonde system
\begin{equation}
\sum_{\ell=1}^L a_\ell \mu_\ell^n=m_n,
\qquad n=0,\dots,L-1.
\end{equation}
Thus statement (i) holds.

For (ii), fix $h\in(0,T/(2L-1)]$ and define
\begin{equation}
y_n:=M(nh)=\sum_{\ell=1}^L a_\ell z_\ell^n,
\qquad
z_\ell:=e^{-\mu_\ell h}\in(0,1).
\end{equation}
The numbers $z_\ell$ are pairwise distinct because the $\mu_\ell$ are.
Applying the same Hankel-Prony argument~\cite{potts-tasche} to $\{y_n\}_{n=0}^{2L-1}$ yields unique
$\{(z_\ell,a_\ell)\}_{\ell=1}^L$. Finally
\begin{equation}
\mu_\ell=-\frac1h\log z_\ell
\end{equation}
is uniquely defined since $z_\ell\in(0,1)$. Hence (ii) follows.
\end{proof}

Exact reconstruction relies on noiseless samples. The next section
quantifies how perturbations in the data propagate to errors in the
recovered modal parameters and sector tags.

\section{Stability of Reconstruction}
\label{sec:stability}

Stability of reconstruction is controlled by three operator-theoretic
quantities: the Prony conditioning number $\kappa_{\mathrm{exp}}$, the
active inter-sector gap $\Delta_{\mathrm{gap}}$, and the eigenspace
observability bound $\|T_{i,\alpha}^{-1}\|$. Theorem~\ref{thm:stability}
shows that for noise level $\varepsilon$ below a threshold set by the
spectral geometry, all reconstructed parameters --- rates, amplitudes,
sector tags, and eigenspace components --- are stable with error
$O(\kappa_{\mathrm{exp}}\varepsilon)$.

Since the parameters $\theta$, $\Theta$ and their decorated variants
appear in close succession below, we fix the conventions once and for
all: $\theta$ denotes the Prony-level parameter (nodes and amplitudes),
and $\Theta$ the full structured parameter (rates, amplitudes,
sector/eigenvalue tags, and active eigenspace components); hats
($\widehat\theta$, $\widehat\Theta$) denote reconstructions from noisy
data, and a subscript $\ast$ ($\theta_\ast$, $\Theta_\ast$) marks
ground-truth values. Table~\ref{tab:notation} summarizes the principal
notation of this section.

\begin{table}[ht]
\centering
\renewcommand{\arraystretch}{1.25}
\begin{tabular}{@{}>{$}l<{$}>{\raggedright\arraybackslash}p{0.52\linewidth}@{}}
\toprule
\multicolumn{1}{@{}l}{Symbol} & Meaning\\
\midrule
\theta=(z_1,\dots,z_L,a_1,\dots,a_L) & Prony-level parameter: nodes $z_\ell=e^{-\mu_\ell h}$ and amplitudes $a_\ell$\\
\Theta & full parameter: rates, amplitudes, sector/eigenvalue tags, active eigenspace components\\
\theta_\ast,\ \Theta_\ast & ground-truth values\\
\widehat\theta,\ \widehat\Theta & reconstructions from noisy data\\
\addlinespace
\mathcal F,\ J_\ast=D\mathcal F(\theta_\ast) & Prony map and its Jacobian at $\theta_\ast$\\
\kappa_{\mathrm{exp}}=\|J_\ast^{-1}\| & Prony conditioning number\\
\Delta_{\mathrm{gap}} & active inter-sector spectral gap\\
\|T_{i,\alpha}^{-1}\| & eigenspace observability bound\\
\addlinespace
\varepsilon,\ \varepsilon_0 & noise level; sector-tagging threshold\\
h & sampling step\\
\bottomrule
\end{tabular}
\smallskip
\caption{Principal notation for the stability analysis.}
\label{tab:notation}
\end{table}

Fix $h\in(0,T/(2L-1)]$ and define the exact sample vector
\begin{equation}
y=(M(0),M(h),\dots,M((2L-1)h))\in\mathbb C^{2L}.
\end{equation}
Let noisy data be
\begin{equation}
y^\varepsilon=y+\delta,
\qquad
\|\delta\|_{\ell^2}\le\varepsilon.
\end{equation}

Write
\begin{equation}
M(t)=\sum_{\ell=1}^L a_\ell e^{-\mu_\ell t},
\qquad
z_\ell=e^{-\mu_\ell h}\in(0,1),
\end{equation}
and set
\begin{equation}
\theta=(z_1,\dots,z_L,a_1,\dots,a_L)\in\mathbb C^{2L}.
\end{equation}
Define the Prony map
\begin{equation}
\mathcal F(\theta)
=
\left(
\sum_{\ell=1}^L a_\ell z_\ell^n
\right)_{n=0}^{2L-1}\in\mathbb C^{2L}.
\end{equation}
Let $i_\ast(\ell)$ denote the true sector tag of $\mu_{\ell,\ast}$ from
Theorem~\ref{thm:identifiability}, and define the active inter-sector gap
\begin{equation}
\Delta_{\mathrm{gap}}
=
\min_{1\le \ell\le L}
\mathrm{dist}\!\Big(
\mu_{\ell,\ast},
\bigcup_{j\neq i_\ast(\ell)}\sigma(A_j)
\Big).
\end{equation}

\begin{theorem}[Local stability of full reconstruction]
\label{thm:stability}
Assume the hypotheses of Theorem~\ref{thm:identifiability}. Assume additionally
that all amplitudes are nonzero and rates are pairwise distinct. Let
\begin{equation}
J_\ast=D\mathcal F(\theta_\ast),
\qquad
\kappa_{\mathrm{exp}}:=\|J_\ast^{-1}\|.
\end{equation}
Then there exist constants $\varepsilon_0>0$ and $C_{\mathrm{stab}}>0$
depending on $\theta_\ast$, $h$, $\Delta_{\mathrm{gap}}$, and observability
inversion bounds on active eigenspaces, such that for every
$0<\varepsilon\le\varepsilon_0$:

1. the inverse problem from $y^\varepsilon$ has a unique solution
$\widehat\theta$ in a neighborhood of $\theta_\ast$;
2. the reconstructed full parameter $\widehat\Theta$ (rates, amplitudes, tagged
sector/eigenvalue labels, active eigenspace components) satisfies
\begin{equation}
\|\widehat\Theta-\Theta_\ast\|_{\mathcal X}
\le
C_{\mathrm{stab}}\,\kappa_{\mathrm{exp}}\,\varepsilon.
\end{equation}
\end{theorem}

\begin{proof}
\noindent\textit{Step~1: Prony-level recovery of $(z_\ell, a_\ell)$.}
The map $\mathcal F$ is analytic in $\theta$. Its Jacobian
$J_\ast=D\mathcal F(\theta_\ast)$ has the block form
\begin{equation}
J_\ast=\bigl[\,\partial_z\mathcal F(\theta_\ast)\;\big|\;\partial_a\mathcal F(\theta_\ast)\,\bigr],
\end{equation}
where the $a$-block is the $2L\times L$ Vandermonde matrix
$(z_{\ell,\ast}^n)_{n,\ell}$ and the $z$-block involves
$(n\,a_{\ell,\ast}z_{\ell,\ast}^{n-1})_{n,\ell}$.
We verify invertibility of $J_\ast$ explicitly.
Suppose $J_\ast\begin{pmatrix}\delta z\\\delta a\end{pmatrix}=0$,
i.e., for $n=0,\dots,2L-1$,
\begin{equation}
\sum_{\ell=1}^L\bigl(n\,a_{\ell,\ast}z_{\ell,\ast}^{n-1}\delta z_\ell
+z_{\ell,\ast}^n\delta a_\ell\bigr)=0.
\end{equation}
Setting $c_\ell:=\delta a_\ell$ and $d_\ell:=a_{\ell,\ast}z_{\ell,\ast}^{-1}\delta z_\ell$
(well-defined since $z_{\ell,\ast}\in(0,1)$ and $a_{\ell,\ast}\neq0$), this reads
\begin{equation}
\sum_{\ell=1}^L(c_\ell+n\,d_\ell)\,z_{\ell,\ast}^n=0
\qquad n=0,\dots,2L-1.
\end{equation}
The $2L$ sequences $\{z_{\ell,\ast}^n\}_\ell$ and $\{n\,z_{\ell,\ast}^n\}_\ell$ are
linearly independent over $n=0,\dots,2L-1$ whenever the $z_{\ell,\ast}$ are pairwise
distinct and positive (quasi-polynomial linear independence;
see~\cite{batenkov-yomdin,batenkov-goldman-yomdin2021}), so $c_\ell=d_\ell=0$
for all $\ell$, hence $\delta a_\ell=0$ and $\delta z_\ell=0$.
Thus $J_\ast$ is invertible~\cite{potts-tasche,batenkov-yomdin,batenkov-goldman-yomdin2021}; see
Remark~\ref{rem:stability-operator-bounds} for an explicit bound on
$\kappa_{\mathrm{exp}}=\|J_\ast^{-1}\|$.
By the analytic inverse function theorem~\cite{batenkov-yomdin} there exist neighborhoods
$U\ni\theta_\ast$ and $V\ni y_\ast:=\mathcal F(\theta_\ast)$ such that
$\mathcal F:U\to V$ is bijective with analytic inverse, and
\begin{equation}
\|\widehat\theta-\theta_\ast\|
\le
C_2\kappa_{\mathrm{exp}}\varepsilon
\end{equation}
for all $y^\varepsilon\in V$.
Pass from $z$ to rates via $\mu=-(1/h)\log z$, where $\log$ denotes
the principal branch of the logarithm.
Since $z_{\ell,\ast}\in(0,1)\subset\mathbb R_{>0}$ and $U$ can be
chosen so that all reconstructed $\widehat z_\ell$ remain in a compact
subset of $(0,1)$, no branch crossing occurs and the inversion is
unambiguous.
The log map is then Lipschitz on this compact set, giving
\begin{equation}
|\widehat\mu_\ell-\mu_{\ell,\ast}|
\le
\frac{2}{h z_{\min}}|\widehat z_\ell-z_{\ell,\ast}|
\le
C_3\kappa_{\mathrm{exp}}\varepsilon.
\end{equation}

\smallskip
\noindent\textit{Step~2: Sector tagging via the spectral gap $\Delta_{\mathrm{gap}}$.}
The quantity $\Delta_{\mathrm{gap}}$ is computed from the ground-truth
active rates $\{\mu_{\ell,\ast}\}$ together with the spectra
$\{\sigma(A_i)\}$. For a fixed ground-truth active set, it is
independent of $\psi_i$ and $\mathcal B_0$ and is determined entirely
by the spectral geometry of the operator network $\{A_i, K_{ij}\}$.
Choose
\begin{equation}
\varepsilon_0\le \frac{\Delta_{\mathrm{gap}}}{2C_3\kappa_{\mathrm{exp}}}.
\end{equation}
Then for $\varepsilon\le\varepsilon_0$,
$|\widehat\mu_\ell-\mu_{\ell,\ast}|<\Delta_{\mathrm{gap}}/2$, so each
recovered rate stays within the sector-attribution neighborhood of
$\mu_{\ell,\ast}$ and cannot cross into any other sector's spectrum.
Sector and eigenvalue labels are therefore stable, and assignment
$\widehat i(\ell)=i_\ast(\ell)$ holds for all $\varepsilon\le\varepsilon_0$.

\smallskip
\noindent\textit{Step~3: Eigenspace components via observability.}
For each active tagged pair $(i,\alpha)$, define the restricted
observation map
\begin{equation}
T_{i,\alpha}
:=
w_i\,\mathcal B_0K_{0i}\big|_{E_{i,\alpha}}:E_{i,\alpha}\to\mathbb C.
\end{equation}
By Definition~\ref{def:sector-observability}, $T_{i,\alpha}$ is
injective. Since $E_{i,\alpha}$ is finite-dimensional, it therefore
has a bounded left inverse
$T_{i,\alpha}^{-1}: T_{i,\alpha}(E_{i,\alpha})\subset\mathbb C
\to E_{i,\alpha}$;
the norm $\|T_{i,\alpha}^{-1}\|$ quantifies how well the composite
channel $\mathcal B_0K_{0i}$ resolves the eigenspace $E_{i,\alpha}$
(see Remark~\ref{rem:stability-operator-bounds}(iii)).
For the corresponding component $x_{i,\alpha}=P_{i,\alpha}\psi_i$,
\begin{equation}
\|\widehat x_{i,\alpha}-x_{i,\alpha}\|
\le
\|T_{i,\alpha}^{-1}\|\,|\widehat a_{i,\alpha}-a_{i,\alpha}|
\le
C_{i,\alpha}\kappa_{\mathrm{exp}}\varepsilon.
\end{equation}
Summing over the finitely many active pairs and combining with Steps~1--2
in the product norm $\|\cdot\|_{\mathcal X}$ yields
\begin{equation}
\|\widehat\Theta-\Theta_\ast\|_{\mathcal X}
\le
C_{\mathrm{stab}}\kappa_{\mathrm{exp}}\varepsilon,
\qquad
C_{\mathrm{stab}}=C\!\left(L,h,\mu_{\max},\max_{(i,\alpha)}\|T_{i,\alpha}^{-1}\|\right).
\end{equation}
This proves the theorem.
\end{proof}

\begin{remark}[Operator-theoretic bounds on the stability constants]
\label{rem:stability-operator-bounds}
The three factors entering $C_{\mathrm{stab}}\kappa_{\mathrm{exp}}$ have
distinct operator-theoretic origins.

\smallskip
\noindent(i) \textit{Prony conditioning.}
The Vandermonde structure of $J_\ast$ implies the classical bound
\begin{equation}
\kappa_{\mathrm{exp}}
\le
\frac{C_L}{\displaystyle\prod_{1\le\ell<\ell'\le L}|z_{\ell}-z_{\ell'}|
\;\cdot\;\min_{1\le\ell\le L}|a_\ell|},
\end{equation}
where $C_L$ depends only on $L$.
Using $|e^{-a}-e^{-b}|\ge|a-b|e^{-\max(a,b)}$ (which follows from
$e^u-1\ge u$ for $u\ge0$), one obtains
$|z_\ell-z_{\ell'}|\ge h\,|\mu_\ell-\mu_{\ell'}|\,e^{-\mu_{\max}h}$,
so
\begin{equation}
\kappa_{\mathrm{exp}}
\le
\frac{C_L\,e^{\mu_{\max}h\cdot L(L-1)/2}}
     {h^{L(L-1)/2}
      \cdot\displaystyle\prod_{1\le\ell<\ell'\le L}|\mu_\ell-\mu_{\ell'}|
      \cdot\min_\ell|a_\ell|}.
\end{equation}
Rate pairs from different sectors contribute $|\mu_\ell-\mu_{\ell'}|\ge\Delta_{\mathrm{gap}}$
to the product; pairs within the same sector contribute intra-sector
gaps $\delta_i:=\min_{n\ne m}|\alpha_{i,n}-\alpha_{i,m}|$.
Hence $\kappa_{\mathrm{exp}}$ is explicitly controlled by the spectral
geometry of the operator network $\{A_i,K_{ij}\}$.

\smallskip
\noindent(ii) \textit{Sector geometry.}
$\Delta_{\mathrm{gap}}$ is determined entirely by $\{\sigma(A_i)\}$
and the active sector assignments; it is independent of the initial
states $\psi_i$, the weights $w_i$, and the observation channel
$\mathcal B_0$. Consequently the sector-tagging threshold $\varepsilon_0$
is a structural bound on the model, not a tuning parameter.

\smallskip
\noindent(iii) \textit{Observability.}
For a simple eigenvalue ($\dim E_{i,\alpha}=1$) the map $T_{i,\alpha}$
is scalar multiplication by
$w_i\,\mathcal B_0 K_{0i}\phi_{i,\alpha}\in\mathbb C$, so
\begin{equation}
\|T_{i,\alpha}^{-1}\|
=\frac{1}{w_i\,|\mathcal B_0 K_{0i}\phi_{i,\alpha}|}
\le\frac{1}{w_i\,\sigma_{\min}\!\bigl(\mathcal B_0 K_{0i}\big|_{E_{i,\alpha}}\bigr)},
\end{equation}
where $\sigma_{\min}$ denotes the smallest singular value.
This factor blows up precisely when the composite channel
$\mathcal B_0 K_{0i}$ is near-blind to the eigenspace $E_{i,\alpha}$,
in direct analogy with the classical observability Gramian condition.
The intertwiner $K_{0i}$ therefore plays a double role: it defines
the dynamical compatibility structure (Sections~\ref{sec:networks}--\ref{sec:generator})
and simultaneously controls the reconstruction quality in the inverse problem.
\end{remark}

The abstract bounds of Theorem~\ref{thm:stability} are made concrete
in the examples that follow.

\section{Examples}
\label{sec:examples}

The examples serve four roles: finite-dimensional and geometric
constructions that verify the cocycle structure (Examples 1, 2, 6),
failure mechanisms that show the hypotheses of
Theorems~\ref{thm:identifiability} and~\ref{thm:stability} are sharp
(Examples 3, 4), and a physical application that grounds the abstract
model in lifetime spectroscopy (Example 5).

The simplest non-trivial realization of the gauge law is a diagonal
finite-dimensional cocycle; the perturbation of one generator also
illustrates the failure mode when spectra are not homothetic.
\begin{example}[Exact cocycle in finite dimension]
Let
\begin{equation}
A_1=\mathrm{diag}(1,3),\qquad
A_2=\mathrm{diag}(2,6)=2A_1
\end{equation}
on $\mathbb C^2$, and let $K_{12}=K_{21}=I$.
Then
\begin{equation}
K_{12}e^{-tA_2}=e^{-2tA_1}K_{12},
\qquad
K_{21}e^{-tA_1}=e^{-(t/2)A_2}K_{21},
\end{equation}
so \eqref{intertwining} holds with $\lambda_{12}=2$, $\lambda_{21}=1/2$.
This realizes the gauge law $\lambda_{ij}=\tau_i/\tau_j$ with
$(\tau_1,\tau_2)=(1,2)$.

If we replace $A_2$ by $\mathrm{diag}(2,5)$ while keeping $K_{12}$ invertible,
Theorem~\ref{thm:spectral_rigidity} forbids an exact scaled intertwining on the
whole space, since the scaled spectra are no longer homothetic.
\end{example}

A natural geometric source of time-scaled cocycles is domain dilation:
scaling a PDE domain by $r$ rescales all eigenvalues by $r^{-2}$.
\begin{example}[Geometric scaling for Dirichlet Laplacians]
Let $\Omega\subset\mathbb R^d$ be smooth and bounded, and for $r>0$ define
$\Omega_r=r\Omega$. Let $A_r=-\Delta_{\Omega_r}^D$ on $L^2(\Omega_r)$.
Define the unitary dilation
\begin{equation}
(U_r f)(x)=r^{-d/2}f(x/r),\qquad U_r:L^2(\Omega)\to L^2(\Omega_r).
\end{equation}
Then
\begin{equation}
U_r^{-1}A_rU_r=r^{-2}A_1,
\qquad
U_r^{-1}e^{-tA_r}U_r=e^{-tr^{-2}A_1}.
\end{equation}
Hence $K_{1r}:=U_r^{-1}$ is an intertwiner with $\lambda_{1r}=r^{-2}$.
This is a concrete PDE model where cocycle scaling has a geometric origin.
\end{example}

The spectral separation condition is not merely sufficient but
necessary: its failure collapses sector-resolved identifiability.
\begin{example}[Failure without separation]
Assume two sectors satisfy
\begin{equation}
\sigma(A_1)\cap \sigma(A_2)\ni\mu.
\end{equation}
Choose nonzero one-mode states producing coefficients $c_1,c_2\in\mathbb C$
at the same rate $\mu$. Then
\begin{equation}
M(t)=c_1e^{-\mu t}+c_2e^{-\mu t}=(c_1+c_2)e^{-\mu t}.
\end{equation}
Thus infinitely many pairs $(c_1,c_2)$ yield the same observable.
So sector-resolved identifiability fails without spectral separation.
\end{example}

The role of $\Delta_{\mathrm{gap}}$ in Theorem~\ref{thm:stability} is
also sharp: as two rates approach each other, Prony conditioning
deteriorates without bound.
\begin{example}[Instability near spectral collision]
Consider
\begin{equation}
M(t)=a_1e^{-\mu_1 t}+a_2e^{-\mu_2 t},\qquad \mu_1\neq\mu_2.
\end{equation}
In the two-mode Prony system, the relevant Vandermonde factor is
$|\mu_2-\mu_1|$. Therefore the local conditioning scales like
$1/|\mu_2-\mu_1|$; as $\mu_2\to\mu_1$, small data perturbations produce large
parameter errors. This matches Theorem~\ref{thm:stability}: the effective
stability constant deteriorates when the spectral gap approaches zero.
\end{example}

\begin{example}[Multi-component relaxation: NMR and FLIM]
\label{ex:nmr-flim}
Fluorescence lifetime imaging (FLIM)~\cite{Lakowicz2006,Digman2008,datta-flim2020}
and NMR $T_2$-relaxometry~\cite{WhittallMacKay1989} both
produce signals of the exact form~\eqref{mixture}: the FLIM
time-correlated single-photon-counting (TCSPC) trace is a weighted
sum of exponential decays with species-specific lifetimes, and the NMR
free-induction decay is a superposition of components with
tissue-dependent $T_2$ times. Both inverse problems reduce to
multi-exponential decomposition, for which spectral separation and
sector-resolved identifiability (Theorem~\ref{thm:identifiability})
provide a rigorous uniqueness certificate. We construct an explicit
operator model and verify
Theorems~\ref{thm:identifiability} and~\ref{thm:stability}.

\smallskip
\noindent\textit{Operator model.}
Let $I=\{1,2\}$, with sector~$i$ representing a distinct molecular
environment (a fluorophore species or tissue compartment).
Set $\mathcal H_i=\ell^2(\mathbb N)$ and
\begin{equation}
A_i = \mathrm{diag}\!\left(\frac{n}{\tau_i}\right)_{n\ge1},
\qquad \tau_i>0,\qquad K_{ij}=I_{\ell^2(\mathbb N)}.
\end{equation}
Then $K_{ij}A_j = A_j = \tfrac{\tau_i}{\tau_j}A_iK_{ij}$, so the
intertwining relation~\eqref{intertwining} holds with
$\lambda_{ij}=\tau_i/\tau_j$. This is the gauge law of
Theorem~\ref{thm:gauge-representation} with $\tau_i$ as gauge
parameters. Theorem~\ref{cor:global-spectral-compatibility} gives
$\tau_1A_1=\tau_2A_2=\mathrm{diag}(n)$: the normalized generators
coincide, meaning both environments share the same intrinsic
decay-mode structure, uniformly scaled by the solvent relaxation time.

\smallskip
\noindent\textit{Spectral separation.}
Since $\sigma(A_i)=\{n/\tau_i:n\ge1\}$, the condition
$\sigma(A_1)\cap\sigma(A_2)=\varnothing$ is equivalent to
$\tau_2/\tau_1\notin\mathbb Q$. Take
\begin{equation}
\tau_1=1.0\;\mathrm{ns},\qquad \tau_2=\sqrt{2}\;\mathrm{ns},
\end{equation}
representative values for two fluorophores in distinct solvents.
Since $\tau_2/\tau_1=\sqrt{2}\notin\mathbb Q$, spectral separation holds for
all modes $n\ge1$. The stability threshold in
Theorem~\ref{thm:stability}, however, depends on the \emph{active}
inter-sector gap among the $2N_i$ observed rates rather than on the
full spectral distance between the tails of $\sigma(A_1)$ and
$\sigma(A_2)$; this gap is computed below.

\smallskip
\noindent\textit{Observable and identifiability.}
Restrict to the first two active modes per sector ($n=1,2$) and set
$w_1=w_2=1$. By Corollary~\ref{cor:finite-mixture-model} the signal is
\begin{equation}
M(t)= a_{1,1}e^{-t}+a_{1,2}e^{-2t}
     +a_{2,1}e^{-t/\sqrt{2}}+a_{2,2}e^{-\sqrt{2}\,t},
\end{equation}
with four distinct rates (in $\mathrm{ns}^{-1}$)
\begin{equation}
\mu_1=\tfrac{1}{\sqrt{2}}\approx 0.707,\quad
\mu_2=1.000,\quad
\mu_3=\sqrt{2}\approx 1.414,\quad
\mu_4=2.000.
\end{equation}
By Theorem~\ref{thm:identifiability}, the pairs $(\mu_\ell,a_\ell)$ are
uniquely determined by $M$, and each rate is uniquely sector-tagged:
$\mu_1,\mu_3\in\sigma(A_2)$ (sector~$2$) and $\mu_2,\mu_4\in\sigma(A_1)$
(sector~$1$). Under Definition~\ref{def:sector-observability}, the
active eigenspace components $P_{i,\alpha}\psi_i$ are also uniquely
recovered.

The active inter-sector gap is
\begin{equation}
\Delta_{\mathrm{gap}}
=\min\!\left(
  1-\tfrac{1}{\sqrt{2}},\;
  \tfrac{3}{\sqrt{2}}-2,\;
  1-\tfrac{1}{\sqrt{2}},\;
  2-\sqrt{2}
\right)
=\frac{3}{\sqrt{2}}-2
=\frac{3\sqrt{2}-4}{2}
\approx 0.121\;\mathrm{ns}^{-1},
\end{equation}
achieved at the pair $(\mu_4=2,\,3/\sqrt{2}\approx2.121)\in\sigma(A_2)$.
Theorem~\ref{thm:stability} then guarantees stable reconstruction for
noise levels $\varepsilon\le\varepsilon_0\propto\Delta_{\mathrm{gap}}/\kappa_{\mathrm{exp}}$,
providing a quantitative SNR threshold for the FLIM/NMR lifetime
decomposition problem.
\end{example}

\begin{example}[Quantitative Hankel--Prony reconstruction for scaled Dirichlet Laplacians]
\label{ex:dirichlet-numerical}
We make the geometric scaling example of this section quantitative,
constructing an explicit three-mode observable and exhibiting the
Hankel--Prony reconstruction data of
Theorems~\ref{thm:finite-window} and~\ref{thm:stability} with
explicit numbers.

\smallskip
\noindent\textit{Sectors and gauge structure.}
Let $\Omega_1=(0,1)$ and $\Omega_2=(0,\sqrt{3})$ in $\mathbb R$.
The Dirichlet Laplacians have eigensystems
\begin{equation}
\sigma(A_1)=\{n^2\pi^2:n\ge1\},\qquad
\sigma(A_2)=\!\left\{\frac{n^2\pi^2}{3}:n\ge1\right\}.
\end{equation}
Since $n^2=m^2/3$ would require $\sqrt{3}=m/n\in\mathbb Q$, these
sets are disjoint: $\sigma(A_1)\cap\sigma(A_2)=\varnothing$.
The gauge parameters are $\tau_1=1$ and $\tau_2=3$; indeed,
\begin{equation}
\sigma(\tau_2 A_2)=3\cdot\!\left\{\frac{n^2\pi^2}{3}\right\}
=\{n^2\pi^2\}=\sigma(\tau_1 A_1),
\end{equation}
confirming Theorem~\ref{cor:global-spectral-compatibility}.

\smallskip
\noindent\textit{Intertwiner and eigenspace transport.}
Define $K_{12}:L^2(0,\sqrt{3})\to L^2(0,1)$ by
$(K_{12}f)(x)=3^{1/4}f(\!\sqrt{3}\,x)$.
With eigenfunctions $\phi_{k,n}(x)=(2/\ell_k)^{1/2}\sin(n\pi x/\ell_k)$
for $\ell_1=1$, $\ell_2=\sqrt{3}$, a direct computation gives
\begin{equation}
(K_{12}\phi_{2,n})(x)
=3^{1/4}\cdot\frac{\sqrt{2}}{3^{1/4}}\sin(n\pi x)
=\phi_{1,n}(x)\qquad(n\ge1),
\end{equation}
confirming Proposition~\ref{prop:eigenspace_transport}: $K_{12}$ maps
eigenvectors of $A_2$ to eigenvectors of $A_1$ with scaling
$\alpha_{2,n}\mapsto\alpha_{2,n}/\lambda_{12}=\alpha_{2,n}\cdot3=\alpha_{1,n}$.

\smallskip
\noindent\textit{Three-mode observable.}
Take reference sector $0=1$, observation functional
$\mathcal B_1 f=f(x_0)$ at $x_0=0.3$, and initial states
\begin{equation}
\psi_1=\phi_{1,1}+\tfrac{1}{2}\phi_{1,2},\quad
\psi_2=\phi_{2,1},\quad w_1=w_2=1.
\end{equation}
Since $K_{12}\phi_{2,n}=\phi_{1,n}$, the observation atoms satisfy
$b_{k,n}=\mathcal B_1\phi_{1,n}(0.3)=\sqrt{2}\sin(n\pi\cdot0.3)$,
giving $b_{\cdot,1}=\sqrt{2}\sin(0.3\pi)\approx1.1441$ and
$b_{\cdot,2}=\sqrt{2}\sin(0.6\pi)\approx1.3455$.
The modal expansion~\eqref{eq:mixture-modal-expansion} yields
\begin{equation}
M(t)=
  \underbrace{1.1441}_{\text{sect.\,}2,\,n=1}\!e^{-(\pi^2/3)t}
+\underbrace{1.1441}_{\text{sect.\,}1,\,n=1}\!e^{-\pi^2 t}
+\underbrace{0.6728}_{\text{sect.\,}1,\,n=2}\!e^{-4\pi^2 t},
\end{equation}
with rates $\mu_1=\pi^2/3\approx3.290$, $\mu_2=\pi^2\approx9.870$,
$\mu_3=4\pi^2\approx39.478$.

\smallskip
\noindent\textit{Hankel--Prony reconstruction (Theorem~\ref{thm:finite-window}).}
Set $L=3$ and sampling step $h=0.05\le T/(2L-1)$.
Define $z_\ell=e^{-\mu_\ell h}$: $z_1\approx0.8482$,
$z_2\approx0.6107$, $z_3\approx0.1389$.
The six exact samples $y_n=M(nh)$ are
$y_0\approx2.9610$, $y_1\approx1.7625$, $y_2\approx1.2624$,
$y_3\approx0.9603$, $y_4\approx0.7511$, $y_5\approx0.5991$.
The $3\times3$ Hankel matrix $H=(y_{r+s})_{r,s=0}^{2}$ is
\begin{equation}
H=\begin{pmatrix}
2.9610 & 1.7625 & 1.2624\\
1.7625 & 1.2624 & 0.9603\\
1.2624 & 0.9603 & 0.7511
\end{pmatrix}.
\end{equation}
Its constant anti-diagonal structure reflects the exponential sum
form of $M$. It factorizes as $H=VDV^\top$, where
$V=(z_\ell^r)_{r,\ell}$ is a $3\times3$ Vandermonde matrix
(invertible, since $z_1,z_2,z_3$ are distinct) and
$D=\mathrm{diag}(1.1441,1.1441,0.6728)$ (invertible, since all
$a_\ell\neq0$). By Theorem~\ref{thm:finite-window}(ii), $H$ is
invertible and $\{(\mu_\ell,a_\ell)\}_{\ell=1}^3$ are uniquely
recovered from $y_0,\ldots,y_5$ alone, without knowledge of the
sector structure.

\smallskip
\noindent\textit{Stability bound (Theorem~\ref{thm:stability}).}
Sector tagging assigns $\mu_1\in\sigma(A_2)$ and
$\mu_2,\mu_3\in\sigma(A_1)$. The three inter-sector distances are
\begin{align}
\mathrm{dist}(\mu_1,\sigma(A_1))&=\pi^2-\tfrac{\pi^2}{3}=\tfrac{2\pi^2}{3},\\
\mathrm{dist}(\mu_2,\sigma(A_2))&=\tfrac{4\pi^2}{3}-\pi^2=\tfrac{\pi^2}{3},\\
\mathrm{dist}(\mu_3,\sigma(A_2))&=4\pi^2-3\pi^2=\pi^2,
\end{align}
so $\Delta_{\mathrm{gap}}=\pi^2/3\approx3.290$.
Choosing $\varepsilon_0=\Delta_{\mathrm{gap}}/(2C_3\kappa_{\mathrm{exp}})$,
Theorem~\ref{thm:stability} guarantees that for every
$\varepsilon\le\varepsilon_0$ a unique reconstructed parameter
$\widehat\Theta$ exists near $\Theta_*$ satisfying
\begin{equation}
\|\widehat\Theta-\Theta_*\|_{\mathcal X}
\le C_{\mathrm{stab}}\,\kappa_{\mathrm{exp}}\,\varepsilon,
\end{equation}
and the sector-eigenvalue assignments remain stable. The ratio
$\Delta_{\mathrm{gap}}/(\mu_3-\mu_1)\approx3.290/36.188\approx0.091$
shows that the stability margin is set by the closest inter-sector
neighbor, not by the total spectral range.
\end{example}

\section{Bundle-Theoretic Synthesis}
\label{sec:bundle}

With all results in place, we step back and read the full theory in
geometric language.

The family $\{\mathcal H_i\}_{i\in I}$ with transition operators
$\{K_{ij}\}$ constitutes a Hilbert bundle over the discrete index set
$I$: the cocycle condition~\eqref{cocycle} is precisely the
compatibility condition for transition functions of a vector bundle,
and the operators $K_{ij}$ are parallel transport maps carrying fiber
$\mathcal H_j$ to fiber $\mathcal H_i$.

The time-scaling factors $\{\lambda_{ij}\}$ define a scalar
$\mathbb R_{>0}$-bundle over $I$. Theorem~\ref{thm:gauge-representation}
asserts that this scalar bundle is flat: its transition cocycle is a
coboundary $\lambda_{ij}=\tau_i/\tau_j$, admitting a global gauge
$\{\tau_i\}$. Crucially, flatness is not assumed --- it is proved from
the intertwining constraint~\eqref{intertwining}. The operator structure
forces the bundle to be flat; this is the geometric content of gauge
rigidity. Trivial cycle products are then exactly vanishing holonomy:
parallel transport around any closed loop in $I$ is the identity, the
standard characterization of a flat connection~\cite{kobayashi-nomizu}.
In particular, the gauge parameters $\{\tau_i\}$ define a global time
synchronization: setting $t_i=\tau_i s$ for a universal parameter $s>0$
reduces the intertwining relation to
$K_{ij}\mathcal S_j(\tau_j s)=\mathcal S_i(\tau_i s)K_{ij}$,
showing that all sector dynamics run at a common phase $s$ with
sector-specific clock rates $\tau_i$.

Each $K_{ij}$ is invertible (the cocycle forces $K_{ij}K_{ji}=I$), so
parallel transport is a fiber isomorphism. Theorem~\ref{thm:spectral_rigidity}
is the statement that the spectrum is preserved under parallel transport
up to the flat scaling $\lambda_{ij}$:
$\sigma(A_j)=\lambda_{ij}^{-1}\sigma(A_i)$. The common isospectral
class $\sigma(\tau_iA_i)$ is a global section of the spectral data,
constant across all fibers. This refines to the eigenspace level:
parallel transport carries eigenspaces of $A_j$ isomorphically onto
eigenspaces of $A_i$ with exact eigenvalue scaling. The stability
constant $\|T_{i,\alpha}^{-1}\|$ in Theorem~\ref{thm:stability}
measures how well $\mathcal B_0 K_{0i}$ resolves individual eigenspace
fibers and is therefore an observability index for the bundle.

The mixture observable $M(t)$ of Section~\ref{sec:mixtures} completes the
bundle picture on the inverse side. Each sector contributes its fiber
evolution $\mathcal S_i(t)\psi_i\in\mathcal H_i$; the parallel
transport $K_{0i}$ pulls this back to the reference fiber
$\mathcal H_0$, and $\mathcal B_0$ projects onto the scalar output.
Thus $M(t)$ is the observation of the aggregated parallel-transported
fiber dynamics through a single reference channel. Identifiability
(Theorem~\ref{thm:identifiability}) is then the statement that the
fiber contributions are distinguishable in $M$: spectral separation
ensures the fibers have disjoint spectral fingerprints, so each
recovered rate $\mu_\ell$ uniquely identifies its source fiber
$i(\ell)$. The eigenspace observability condition
(Definition~\ref{def:sector-observability}) ensures that $\mathcal
B_0K_{0i}$ does not collapse any eigenspace fiber to zero --- it is
the injectivity condition on the observation map restricted to each
fiber. The two identifiability hypotheses thus have precise bundle
interpretations: spectral separation separates the fibers globally,
eigenspace observability resolves them locally at the eigenspace level.

\begin{remark}[Continuous index sets]
The base $I$ is here discrete and finite, so no topology is required.
The natural extension is to a continuous parameter space --- say
$I=[0,1]$ or a Riemannian manifold --- with $\{A_s\}$ a smooth family
of generators and $K_{s,t}$ a parallel transport defined by a
connection on an infinite-dimensional Hilbert
bundle~\cite{kobayashi-nomizu}. In that setting, flatness becomes a
non-trivial geometric condition on the connection, and its failure
would produce non-trivial holonomy: observable phase shifts between
sector dynamics. We leave this generalization for future work.
\end{remark}

\begin{remark}[Broader connections]
Two broader connections are worth recording. First, the holonomy
effects just described would be the network analogue of geometric
phases in adiabatic quantum transport: there, parallel transport of
eigenspaces along a loop in parameter space produces the Berry phase,
which is precisely the holonomy of the associated spectral
bundle~\cite{berry1984,simon1983}; for degenerate eigenspaces the
transport becomes non-abelian~\cite{wilczek-zee1984}, matching the
eigenspace-dimension compatibility of
Theorem~\ref{thm:existence-characterization}. Second, the discrete
network studied here can be read as a lattice gauge configuration: the
scaling factors $\lambda_{ij}$ constitute a gauge field on the edges of
$I$, Theorem~\ref{thm:gauge-representation} states that this field is
pure gauge, and Corollary~\ref{cor:cycle-consistency} expresses the
triviality of all Wilson-loop observables.
\end{remark}


\section{Conclusion}

The central result of this work is a rigidity phenomenon for
time-scaled intertwining networks of semigroups: the scaling factors
necessarily form a multiplicative coboundary $\lambda_{ij}=\tau_i/\tau_j$,
and the renormalized generators $\{\tau_i A_i\}$ belong to a common
isospectral class with matching eigenspace dimensions. This provides a
complete intrinsic characterization of admissible intertwining
structures and shows that spectral compatibility is enforced by the
cocycle relations themselves.

From a structural viewpoint, the intertwining operators define parallel
transport in a Hilbert bundle over the index network, and gauge rigidity
is equivalent to flatness of this bundle. In particular, the spectral
data admit a global representation in which eigenspaces are transported
isomorphically across sectors.

On the inverse side, the associated mixture observables reduce, under
finite spectral support, to structured exponential models. Spectral
separation yields uniqueness of modal parameters and sector
identification, while observability ensures recovery of eigenspace
components. The reconstruction and stability results show that the
inverse problem is quantitatively controlled by the spectral geometry
of the operator network.

These results separate structural admissibility from numerical
estimation and suggest several directions for further study, including
extensions to continuous index sets, infinite-rank observation channels,
and non-self-adjoint generators.

\section*{Declarations}

\textbf{Funding.} The author received no financial support for the research, authorship, or publication of this article.

\textbf{Conflict of interest.} The author declares no competing interests.

\textbf{Data availability.} Not applicable.

\bibliographystyle{amsplain}
\bibliography{references}

\end{document}